\documentclass[a4paper,11pt]{amsart}
\usepackage{amssymb,amsmath,amsthm,enumitem}
\usepackage{amscd}
\numberwithin{equation}{section}
\usepackage{epsfig}
\usepackage{amsfonts,amssymb,amsopn}
\usepackage{mathrsfs}
\usepackage{verbatim}
\usepackage{color}
\usepackage{dsfont, hyperref}

\newtheorem{theorem}{Theorem}[section] 
\newtheorem{lemma}[theorem]{Lemma}  
\newtheorem{proposition}[theorem]{Proposition}  
 
\newtheorem{definition}[theorem]{Definition} 
 
\newtheorem{question}[theorem]{Question} 
\newtheorem{conjecture}[theorem]{Conjecture}

\def\bR{\mathbb{R}}

\def\FI{\mathcal{F}_I}

\def\Fq#1{\mathbb{F}_q^{#1}}
\def\norm#1{\left\lVert#1\right\lVert}
\def\pnorm#1#2{|#2|_{#1}}
\def\knorm#1#2{\lVert#2\rVert_{#1}}

\def\CG{\mathcal{G}_1}

\title{The Mattila-Sj\"{o}lin problem for the $k$-distance over a finite field}

\author{Daewoong Cheong}
\address{Chungbuk National University, Department of Mathematics, Chungdae-ro 1, Seowon-Gu, Cheongju City, Chungbuk 28644, Korea}
\email{daewoongc@chungbuk.ac.kr}

\author{Hunseok Kang}
\address{College of Engineering and Technology, American University of the Middle East, Kuwait}
\email{hunseok.kang@aum.edu.kw}

\author{Jinbeom Kim}
\address{Chungbuk National University, Department of Mathematics,
Chungdae-ro 1, Seowon-Gu, Cheongju City, Chungbuk 28644, Korea}
\email{jinbeom337@chungbuk.ac.kr}

\begin{document}

\begin{abstract}
Let $\Fq{d}$ be a $d$-dimensional vector space over a finite field $\Fq{}$ with $q$ elements. For $x\in \Fq{d}$, let $\knorm{}{x} = x_1^2+\cdots +x_d^2$. By abuse of terminology, we shall call $\knorm{}{\cdot}$ a norm on $\Fq{d}$. For a subset $E\subset \Fq{d}$, let $\Delta(E)$ be the distance set on $E$ defined as \[\Delta(E):=\{\knorm{}{x-y} : x, y \in E \}.\]  The Mattila-Sj\"{o}lin problem seeks the smallest exponent $\alpha>0$ such that  $ \Delta(E) =\Fq{}$ for all subsets $E \subset \Fq{d}$ with $|E| \geq Cq^\alpha$.
 In this article, we consider this problem for a variant of this norm, which generates a smaller distance set  than the norm $\knorm{}{\cdot}.$   Namely,  we replace the norm $\knorm{}{\cdot}$ by the so-called $k$-norm  $(1 \leq k \leq d)$, which can be viewed as a kind of deformation of $\knorm{}{\cdot}$. To derive our  result on the Mattila-Sj\"{o}lin problem for the $k$-norm, we use a combinatorial method to analyze various summations arising from the discrete Fourier machinery. Even though our distance set is smaller than the one in the  Mattila-Sj\"{o}lin problem, for some $k$ we still obtain the same result as that of Iosevich and Rudnev \cite{IR07}, which deals with the Mattila-Sj\"{o}lin problem. Furthermore, our result is sharp in all odd dimensions.
\end{abstract}

\maketitle

\section{Introduction}
In geometric measure theory, the Falconer distance problem asks for a minimal Hausdorff dimension $\dim_{\mathcal{H}}(E)$ of compact subsets $E \subset \mathbb{R}^d $ for which  the distance set  
\[
\Delta(E) := \{ \knorm{}{x-y} : x, y \in E \}
\]  
has positive Lebesgue measure, where $\knorm{}{\cdot}$ denotes the Euclidean norm on $\bR^d$. In 1985, Falconer \cite{Fa85} conjectured that  for each  compact subset $ E \subset \mathbb{R}^d $ ($ d \geq 2 $),  if  $\dim_{\mathcal{H}}(E) > d/2$, then the distance set $\Delta(E)$ has positive Lebesgue measure. Falconer then proved that this is the case  when $ \dim_{\mathcal{H}}(E) > (d+1)/2 $. While some progress has been made, the question remains open in all dimensions (for example, see references \cite{Bo94, Wo99, Er05, DGOWWZ18, GIOW19, DZ19, DIOWZ21, DORZ23}).

Also, Mattila and Sj\"{o}lin \cite{MS99} obtained a stronger result that if $ \dim_{\mathcal{H}}(E) > (d+1)/2 $, then $ \Delta(E) $  has a non-empty interior as a subset of $\mathbb R$. Note that the lower bounds in both results are far from the conjectured dimension threshold $d/2$.

The Falconer distance problem can also be considered a continuous version of the Erd\H{o}s distinct distances problem, which asserts that large finite point sets must have many distinct distances. We refer to references \cite{BKT}, \cite{Er46}, \cite{GK10}, \cite{KT04}, and \cite{SV08} for a more precise definition of the Erd\H{o}s distinct distance problem, as well as the known results and conjecture.

On the other hand, Iosevich and Rudnev \cite{IR07} considered a discrete analogue of the Falconer distance problem over a finite field, 
which in the literature is referred to as the Erd\H{o}s-Falconer distance problem (EF problem). Let us introduce the EF problem. For  $x, y\in \Fq{d}$, a distance between $x$ and $y$ is defined by
\begin{equation*}
\|x - y\| := (x_1 - y_1)^2 + \cdots + (x_d - y_d)^2.
\end{equation*}
By abuse of terminology, we shall call this a standard distance on $\Fq{d}$. For a subset $E \subset \Fq{d}$, the  distance set $\Delta(E)$ is given by
\begin{equation}\label{Emetric}
\Delta(E) := \{\|x - y\| : x, y \in E \}.
\end{equation}

The EF problem asks for the smallest exponent $\alpha>0$ for which $|\Delta(E)| \gtrsim q$ whenever $|E| \geq Cq^\alpha$ for a sufficiently large constant $C$. Here, and throughout this paper, $A \gtrsim B$ means $A \geq cB$ where $0 < c \leq 1$ denotes some constant independent of $q$.

There is a stronger version of the EF problem, referred to as the Mattila-Sj\"{o}lin problem (MS problem), which can be stated as follows:
\begin{question}[The MS problem]\label{Q1}
    Determine the smallest exponent $\beta > 0$ such that if $E\subset \mathbb F_q^d$ with $|E| \geq Cq^{\beta}$ for a sufficiently large $C$, then $\Delta(E) = \Fq{}$.
\end{question}

Let us first list some results for the MS problem. Iosevich and Rudnev \cite{IR07} showed that $ \Delta(E) = \mathbb{F}_q $  whenever $ |E| > 2q^{(d+1)/2}$. Since then, much effort was made by many researchers to lower the exponent $\frac{d+1}{2}$. Hart et al. \cite{HIKR10} demonstrated that the exponent $\frac{d+1}{2}$ due to Iosevich and Rudnev \cite{IR07} is indeed sharp for the specific odd dimensions $d$ and under certain assumptions on $\mathbb F_q$.  In addition, they showed that  the exponent $d/2$ for the EF problem cannot be lowered for all even dimensions $d$, which implies that the exponent $\alpha$ should be greater or equal to $\frac{d}{2}$ in this case. Very recently, Iosevich, Koh and Rakhmonov \cite{FIK24} provided another examples to show that for all $d\geq 3,$ the exponent $\alpha$ for EF problem should be greater or equal to $\frac{d}{2}$ (resp. $\frac{d+1}{2}$) for even (resp. odd) $d.$

Let us review some known results for the EF problem. For $d=2$, Chapman et al. \cite{CEHIK10} obtained the exponent $ 4/3 $ for the EF problem, whereas it was shown in \cite{MP19} that the exponent cannot be lowered beyond $4/3$ for the MS problem. This result was obtained by utilizing a relationship between the EF problem and the Fourier restriction problem, and it is the first result to show that the optimal result $(d+1)/2$ in odd dimensions can be improved in even dimensions. In the case where $\mathbb F_q$ is a prime field, the result of $4/3$ has been improved to $5/4$ by Murphy,  Petridis,  Pham,  Rudnev and  Stevens \cite{MPPRS20}. However, in even higher dimensions $d\ge 4,$ there is no known result that improves the result of $(d+1)/2$  for the EF problem.  

The main purpose of this paper is to introduce a modified distance, which generates a smaller distance set than the standard distance does, and show that the distance result of \((d+1)/2\) can still be achieved under this new distance.

Now we present our distance problem and main result.
For $ x = (x_1, x_2, \cdots, x_d) \in \mathbb{F}_q^d $, let $\mathcal{Z}(x)$ denote the number of zero coordinates of $ x $. Fix $k \in \{1, 2, \cdots, d\}$. Define a $k$-norm $\|\cdot\|_{k}$ on $\Fq{d}$ as
\begin{equation}\label{kdst}
\|x\|_{k} = 
\begin{cases}
\|x\| & \text{if } 0 \leq \mathcal{Z}(x) \leq k-1, \\
0 & \text{otherwise.}
\end{cases}
\end{equation}
Note that if $k=d$, two norms coincide with each other; 
$\knorm{}{\cdot} = \knorm{d}{\cdot}$. Thus the $k$-norm $\knorm{k}{\cdot}$ may be considered a deformation of the standard norm $\knorm{}{\cdot}$ such that  it is more deformed as $k$ gets smaller.
The distance induced by $k$-norm shall be called a $k$-distance.
For a subset $E\subset \mathbb F_q^d$, define a $k$-distance set 
\[
D_k(E) = \{\|x - y\|_{k} : x, y \in E\}.
\]
With this newly defined distance set, one can ask the following  natural question concerning a distance-related problem.
\begin{question} [The MS problem for $k$-distance sets]\label{Q2}
    What is the smallest exponent $\gamma > 0$ such that whenever $E\subset \mathbb F_q^d$ with $|E| \geq Cq^{\gamma}$ for a sufficiently large $C$, we have $D_k(E) = \Fq{}$?
\end{question}
Here, we observe that the MS problem for \(k\)-distance sets in Question~\ref{Q2} can be viewed as a significantly stronger version of the original MS problem posed in Question~\ref{Q1}, in the sense that the equality \( D_k(E) = \mathbb{F}_q \) is  less likely to be achieved than the equality \( \Delta(E) = \mathbb{F}_q \) since  $D_k(E) \subset \Delta(E)$ for all $k = 1,2,\cdots,d$.
 This suggests that for odd dimensions \( d \), the smallest possible exponent \( \gamma \) for Question~\ref{Q2} cannot be smaller than the optimal result \( (d+1)/2 \) of the MS problem in Question~\ref{Q1}, and for even dimensions, it cannot be smaller than the conjectured result \( d/2 \) of the same problem. Moreover, for each integer \( k = 1, 2, \ldots, d \), if we take \( E = \mathbb{F}_q^{d-k} \times \{ \mathbf{0}\} \subset \mathbb{F}_q^d \), then \( |E| = q^{d-k} \), yet \( D_k(E) = \{ 0 \} \neq \mathbb{F}_q \). From this example, we can see that the smallest exponent \( \beta \) for Question \ref{Q2} cannot be less than \( d-k \).  Therefore, one can propose the following conjecture on the MS problem for $k$-distance sets.
\begin{conjecture}\label{mcon}  Let $E\subset \mathbb F_q^d$, $d\ge 2,$ and  assume that $C$ is a sufficiently large constant independent of $q.$  Then for each integer $k=1,2, \ldots, d,$  the following statements hold:

\begin{enumerate}[label=\textup{(\roman*)}, ref=\textup{(\roman*)}]
\item \label{oddconj} If $d$ is odd and $|E|\ge C  q^{\max\{(d+1)/2, ~{d-k}\}},$   then $D_k(E) = \mathbb{F}_q.$
\item \label{evenconj}  If $d$ is even and $|E|\ge C  q^{\max\{d/2, ~{d-k}\}},$   then $D_k(E) = \mathbb{F}_q.$
\end{enumerate}
\end{conjecture}

As our main result,  for odd dimension case we establish Conjecture \ref{mcon}-\ref{oddconj};  for the even dimension case we obtain a bit weaker result than in the conjecture.  More precisely, our result is stated as follows.
\begin{theorem}\label{first}  
If $E\subset \mathbb F_q^d$, $d\ge 2,$ with $|E|\ge C  q^{\max\{(d+1)/2, ~{d-k}\}},$  then  $ D_k(E) = \mathbb{F}_q $.
\end{theorem}

Given the paragraph subsequent to Question \ref{Q2}, our result should be compared with that a result for the MS problem, e.g., that of  Iosevich and Rudnev \cite{IR07}.  
 Notice that our exponent in the assumption of Theorem \ref{first} can be divided into two cases
$$\max\left\{(d+1)/2, ~{d-k}\right\} =\left\{\begin{array}{ll}
(d+1)/2 & \text{if }  (d-1)/2 \le k \le d, \\
d-k & \text{if } 1\le k<(d-1)/2. \end{array} \right.$$
For  $1\le k< (d-1)/2,$ the exponent $d-k$ in our result is larger than $(d+1)/2$, the exponent of  Iosevich and Rudnev, as expected from the paragraph subsequent to Question \ref{Q2}. However, for  $k\ge (d-1)/2$, our exponent $ (d+1)/2$ is as small as  the exponent of Iosevich and Rudnev \cite{IR07}.
That is, even under the weaker condition   $D_k(E) \subset \Delta(E)$ in our case, we obtain the same result as that of Iosevich and Rudnev. 

From the point of view of a deformation of the standard distance, this comparison tells us that the result of the original MS problem remains the same under the ``small" deformation, which corresponds to $k$ with $ (d-1)/2 \le k \le d $.

In the course of proving this, we adopt an analytic approach based on the Fourier analysis machinery, a key tool used in previous studies of the MS problem. However, unlike in the original MS problem, there arises a difficulty in proving  Theorem \ref{first}, namely we encounter  many terms in taking sums that are difficult to control. To overcome this difficulty, we introduce a new elimination method based on combinatorial ideas (see, for example, the proof of Lemma \ref{main}, which constitutes the core of the proof of Theorem \ref{first}).

This paper  is organized as follows. In Section \ref{Ch2}, we give preliminaries which includes some basics on the Gauss sums, 
the orthongality of characters and the discrete Fourier transforms. In Section \ref{Ch3}, we introduce the $k$-distance and compute the Fourier transform on the sphere with respect to the $k$-distance. In Section \ref{Ch4}, we give a proof of Theorem~\ref{first}.
In Section~\ref{Ch5}, we carry out a technical computation of Lemma~\ref{main},  which will be used in the proof of the main theorem in Section~\ref{Ch4}.

\subsection*{Acknowledgements}
The first named author was supported by a funding for the academic research program of Chungbuk National University in 2025, and the National Research Foundation of Korea (NRF-2021R1I1A3049181).

\section{Preliminaries}\label{Ch2}
In this section, we list some basics on the Gauss sum and the Discrete Fourier transform machinery used in \cite{IR07} $($see also \cite{KS13} and \cite{CKPS}$)$. 

\subsection{Gauss sum}
An additive (resp. a multiplicative) character of $\Fq{}$ is a group homomorphism from  the group $\Fq{}$ (resp.  $\Fq{*}$) to the unit circle $S^1$ on the complex plane. The following two characters are ubiquitous in this paper.

For each $b\in \Fq{},$ let  
$\chi_b$ : $\mathbb{F}_q\rightarrow S^1$ be the additive character defined by 
\[
	\chi_b(c) = e^{2\pi i \mathrm{Tr}(bc) / p},
\]
where $p$ is the characteristic of $\Fq{}$, and $\mathrm{Tr}(\alpha)=\mathrm{Tr}_{\Fq{}/\mathbb{F}_p}(\alpha)$ denotes the absolute trace of $\alpha$ (see \cite[Definition 2.22]{LN97}).

Let $\eta:\Fq{*}\rightarrow S^1$ be the multiplicative character defined by
$\eta(c)= 1$ if $c$ is a square and $1$ otherwise. 

Characters of a finite abelian group $G$ satisfy the  orthogonality property below which will be frequently used  for $G=\mathbb F_q $ and $\mathbb F_q^*$ in our computation. 
\begin{proposition}[\textbf{Orthogonality of characters}]
\label{Orthogonality}
Let $G$ be a finite abelian group and $\phi$ a character of $G$. Then we obtain
\[
\sum_{g \in G} \phi(g) =
\begin{cases}
0 & \text{if } \phi \text{ is nontrivial}, \\
|G| & \text{otherwise}.
\end{cases}
\]\end{proposition}

\begin{proof}
See  \cite[Example 5.10]{LN97}.
\end{proof}

Now we review the basic facts about the Gauss sum, which plays a crucial role in the proof of Theorem \ref{first}.

\begin{definition}[\textbf{Gauss sum}]\label{GaussSumDf}
Let $\psi$ (resp. $\chi$)  be a multiplicative (resp.  an additive) character on $\mathbb{F}_q^*$ (resp. $\Fq{}$). Then, the Gauss sum of $\psi$ and $\chi$ is defined by
\begin{equation}
 \label{Gs}
 \mathcal{G}(\psi, \chi) = \sum_{c \in \mathbb{F}_q^*} \psi(c)\chi(c).
\end{equation}
\end{definition}
For $a \in \mathbb{F}_q$, let $\mathcal{G}_a$ be the Gauss sum of $\eta$ and $\chi_a$, i.e., $\mathcal{G}_a = \mathcal{G}(\eta, \chi_a)$.

\begin{proposition}[Theorem 5.15, \cite{LN97}]\label{ExGa} The standard Gauss sum $\mathcal{G}_1$ can be explicitly computed as follows.
\begin{equation}
 \mathcal{G}_1 =
 \begin{cases}
 (-1)^{s-1}q^{\frac{1}{2}} &\text{if } ~~ p \equiv 1 \pmod 4,\\
 (-1)^{s-1}\mathbf{i}^sq^{\frac{1}{2}} &\text{if }~~ p \equiv 3 \pmod 4,
 \end{cases}
\end{equation}
where $\mathbf{i} = \sqrt{-1}$, and $s$ is the integer with  $p^s = q$.
\end{proposition}
Squaring the $\CG$, we can easily see that
\begin{equation}\label{GSQETA}\mathcal{G}_1^2=\eta(-1)q, \, \, \, \mathrm{and}\, \, \, \, |\mathcal{G}_1|=\sqrt{q}.\end{equation}

The following lemma can be derived from the definition of Gauss sums. For the readers' convenience, we provide its proof.
%\newpage
\begin{lemma}\label{GaussSum}
For $a, b\in\mathbb F_q$ with $a\ne 0,$ and $v\in\mathbb F_q^d,$ we have
\begin{enumerate}[label=\textup{(\arabic*)}]
\item
\label{GaussSum1}
$\displaystyle \sum_{s\in\mathbb F_q}\chi_1(as^2)=\eta(a)\mathcal{G}_1$,
\item

\label{GaussSum2}
$\displaystyle \sum_{s\in \mathbb F_q}\chi_1(as^2+bs)
=\eta(a)\mathcal{G}_1\chi_1\left(-\frac{b^2}{4a}\right),$

\item
\label{GaussSum3}
$\displaystyle \sum_{u\in\mathbb{F}_q^d}
 \chi_1\left(a\|u\|+v\cdot u\right)=\eta^d(a)  \mathcal{G}_1^d  \chi_1\left(\frac{||v||}{-4a}\right).$

\end{enumerate}
\end{lemma}
\begin{proof}

To demonstrate \ref{GaussSum1}, we observe that
\begin{equation*}
\sum_{s\in\mathbb F_q}\chi_1(as^2)
=1+\sum_{s\in\mathbb F_q^*}\chi_1(as^2).
\end{equation*}

As $(-s)^2=s^2$ for any $s\in\mathbb F_q^*$ we can perform a change of variables, setting $s^2 = t$. Consequently,
\begin{align*}
1+\sum_{s\in\mathbb F_q^*}\chi_1(as^2)
&=1+2\sum_{\substack{t\in\mathbb F_q^*
 \\\emph{:t is a square}}}\chi_1(at) = 1+\sum_{t\in\mathbb F_q^*}\chi_1(at)(\eta(t)+1)\\
&=1+\sum_{t\in\mathbb F_q^*}\chi_1(at)
 +\sum_{t\in\mathbb F_q^*}\chi_1(at)\eta(t) = \sum_{t\in\mathbb F_q^*}\chi_1(at)\eta(t).
\end{align*}
The last equality arises from the orthogonality of characters for $\chi_1$. Now, employing a change of variables $t = a^{-1}\theta$ and the relation $\eta(a)=\eta(a^{-1})$, we obtain \begin{equation*}
\sum_{t\in\mathbb F_q^*}\chi_1(at)\eta(t)
=\sum_{\theta\in\mathbb F_q^*}\eta(a)\chi_1(\theta)\eta(\theta)=\eta(a)\mathcal{G}_1.
\end{equation*}
Statements \ref{GaussSum2} and \ref{GaussSum3} stem from \ref{GaussSum1} by completing the square and utilizing a change of variables.
\end{proof}

\begin{definition}
    
For a nontrivial additive character  $\chi$ and  nonzero elements $a,b \in \Fq{}$, the Kloosterman sum is defined as
\[K(\chi;a,b) := \displaystyle\sum_{s \in \Fq{*}} \chi(as + bs^{-1}).\]
\end{definition}
A proof of the following result can be found in \cite{We48}. 
\begin{proposition}\label{KLsum}
Let $\chi, a, b$ be as in the above. Then, we have
\begin{equation}\label{KSTM}
	|K(\chi;a,b)| \leq 2\sqrt{q}.
\end{equation}
\end{proposition}

\subsection{Discrete Fourier transform} 
For a function $f : \Fq{d} \rightarrow \mathbb{C}$, the Fourier transform $\widehat{f}$ of $f$ is defined as
\begin{equation}\label{FT}
\widehat{f}(m) := q^{-d}\sum_{x \in \Fq{d}} f(x) \chi (-x \cdot m).
\end{equation}

Here, and throughout the paper,  we will use the notation $\chi$ to denote a fixed non-trivial additive character of $\mathbb F_q,$  whose choice will not affect our results. 

The discrete analogue of the Fourier inversion theorem and the Plancherel theorem is as follows.
\begin{equation}
f(x) := \sum_{m \in \Fq{d}}\chi(m \cdot x)\widehat{f}(m),
\end{equation}
\begin{equation}
\sum_{m \in \Fq{d}} \left|{\widehat{f}(m)}\right|^2 = q^{-d}\sum_{x \in \Fq{d}} \left|f(x)\right|^2.
\end{equation}

For a given set $E$, let $I_E(x)$ be the characteristic function:
\begin{equation*}
 I_E(x) :=
 \begin{cases}
 1 &~~x\in E, \\
 0 &~~ x\notin E.
 \end{cases}
\end{equation*}
For simplicity, we write $E(x)$ for $I_E(x)$ if there is no confusion. For $E \subset \Fq{d}$, using the Plancherel theorem, we can easily see that
\begin{equation}\label{PL1}
\sum_{m \in \Fq{d}}\pnorm{}{\widehat{E}(m)}^2 = q^{-d}|E|,
\end{equation}
and it is clear by the definition of the Fourier transform that
\begin{equation}\label{PL2}
\widehat{E}(0,\cdots,0) = q^{-d}|E|.
\end{equation}

\section{The Fourier transform on the sphere w.r.t. the $k$-norm} \label{Ch3}
Let $k$ be an integer with $1\le k\le d.$ For each $t\in \mathbb F_q,$  we define the sphere of radius $t$ with respect to the $k$-norm as follows:
\[S_{k}^t := \{x \in \Fq{d} : \knorm{k}{x} = t\}.\]
Simply, $S_{k}^t$ shall be referred to as a sphere (of radius $t$), without reference to the $k$-norm, if there is no confusion.
In this section, we mostly compute $\widehat{S_{k}^t}$, the Fourier transform on the  sphere $S_{k}^t,$ which will be used in the proof of Theorem \ref{first}.
To this end, we shall first carry out preliminary computation. Let us begin by giving notations.

\subsection{Notation}
For each non-negative integer $d\ge 1,$ let $[d] := \{1,\dots, d\}$. For a subset $I \subset [d]$, we define $\mathcal{F}_I := \{x \in \Fq{d} : x_i \neq 0 \iff i \in I\}$. By convention we set $\mathcal{F}_{\emptyset}
 = \{(0,\ldots,0)\}$.  Then $\Fq{d}$ is a disjoint union of all the `slices' $\mathcal{F}_I$, i.e., \begin{equation}\label{Fslice}
     \Fq{d} = \bigsqcup_{I\subset [d]}\mathcal{F}_I.
 \end{equation}
For a non-empty subset $I = \{i_1, i_2, \ldots, i_e\} \subset  [d]$ with $1\le i_1<i_2< \cdots < i_e\le d$ and $x\in \Fq{d},$  we define $x_I=(x_{i_1},\ldots, x_{i_e})\in \mathbb F_q^e$. In particular, we adopt the convention that   $x_{\emptyset}=(0,\ldots, 0)\in \mathbb F_q^d.$

Recall that for $y \in \Fq{d}$, $\mathcal{Z}(y)$ denotes the number of zero coordinates of $y.$ For example
if $x\in \mathcal{F}_I$, then $\mathcal{Z}(x)= d-|I|$. For $\alpha \in [d]\cup \{0\}$, we define $N_\alpha := \{x \in \Fq{d} : \mathcal{Z}(x) = \alpha\}$. 
Thus, $N_\alpha$ can be written as
\begin{equation}\label{Na}
N_\alpha = \bigsqcup_{\substack{{I \subset [d]}\\ {:\, |I| = d- \alpha}}} \FI.
\end{equation}

\subsection{Preliminary computation}
The following lemma will be used to compute $\widehat{S_{k}^t}(m).$
\begin{lemma} \label{Lem3.1} Let $m = (m_1, m_2, \cdots, m_d)\in \Fq{d}$. For each  $\alpha \in [d]\cup \{0\},$  we have
\begin{equation*}\label{m}
    \displaystyle{\sum_{x \in N_\alpha} \chi(s\norm{x} - m\cdot x)} = \begin{cases} 		\displaystyle{\sum_{\substack{{I \subset [d]}\\{:|I| = d-\alpha}}} \prod_{i \in I}\left(\eta(s)\mathcal{G}_1\chi\left(-\frac{{m_i}^2}{4s}\right)-1\right)} & \text{if } s\neq 0, \\ 
         \displaystyle{\sum_{\substack{{I \subset [d]}\\{:|I| = d-\alpha}}}(q-1)^{\mathcal{Z}(m_{I})}(-1)^{d-\alpha-\mathcal{Z}(m_I)}} & \text{if } s = 0.
     \end{cases}
\end{equation*}
Here, if $\alpha=d$, then $|I| = d-\alpha=0$ and so $I=\emptyset$, in which case the summation on the right-hand side for $s\ne 0$ is set to be $1.$ 
\end{lemma}

\begin{proof}  Fix an element $\alpha \in [d]\cup \{0\}.$  It follows that

\begin{align}\label{sm}
\sum_{x \in N_\alpha} \chi(s\norm{x} - m\cdot x) 
& =\sum_{\substack{{I \subset [d]}\\{:|I| = d-\alpha}}}\sum_{x \in \mathcal{F}_I}\chi(s||x_I||-m_I \cdot x_I) \notag \\
& =\sum_{\substack{{I \subset [d]}\\{:|I| = d-\alpha}}}\prod_{i \in I}\left[\left(\sum_{u \in \Fq{}}\chi(su^2-m_iu)\right)-1\right].
\end{align}
The first equality follows from \eqref{Na}, and the second equality just unravels the notation of $\mathcal{F}_I$.
We consider two cases: $s \neq 0$, and $s =0$. For the case where $s\neq 0$, applying Definition \ref{GaussSumDf} and Lemma \ref{GaussSum}, we have
\begin{align*}
\sum_{\substack{{I \subset [d]}\\{:|I| = d-\alpha}}}\prod_{i \in I}\left[\left(\sum_{u \in \Fq{}}\chi(su^2-m_iu)\right)-1\right] 
= \sum_{\substack{{I \subset [d]}\\{:|I| = d-\alpha}}}\prod_{i \in I}\left(\eta(s)\mathcal{G}_1\chi\left(-\frac{m_i^2}{4s}\right)-1\right).
\end{align*}
For the case where $s=0$, \eqref{sm} is equal to
\begin{align*}
\sum_{\substack{{I \in [d]}\\:{|I| = d - \alpha }}}\prod_{\substack{{i \in I}\\{:m_i \neq 0}}} & \left[\left(\sum_{u \in \Fq{}} \chi(-m_iu)\right)-1\right] \prod_{\substack{{j \in I}\\{:m_j =0}}}\left[\left(\sum_{u \in \Fq{}}\chi(-m_ju)\right)-1\right]\\
& = \sum_{\substack{{I \subset [d]}\\{:|I| = d-\alpha}}}(q-1)^{\mathcal{Z}(m_{I})}(-1)^{d-\alpha-\mathcal{Z}(m_I)}.
\end{align*}
This completes the proof.
\end{proof}
\subsection{Computation of $\widehat{S_{k}^t} (m)$}
We are ready to compute the following Fourier transform on the sphere $\widehat{S_{k}^t} (m)$.
\begin{proposition}\label{FSF} Let $t\in \mathbb F_q$. Then, for any $m\in \mathbb F_q^d,$ the Fourier transform $\widehat{S_{k}^t} (m)$ of the sphere can be written
$$\widehat{S_{k}^t} (m)=q^{-d-1} \bigg(A(m,t)+B(m)\bigg),$$
where 
\begin{align*}
    A(m,t) &=\sum_{\alpha = 0}^{k-1}\sum_{\substack{{I \subset [d]}\\{:|I| = d-\alpha}}}\sum_{s\in\Fq{*}}\chi(-st)\prod_{i \in I}\bigg(\eta(s)\mathcal{G}_1\chi\left(-\frac{{m}_i^2}{4s}\right)-1\bigg), \\
    B(m) &= \sum_{\alpha = 0}^{k-1}\sum_{\substack{{I \subset [d]}\\{:|I| = d-\alpha}}}(q-1)^{\mathcal{Z}(m_{I})}(-1)^{d-\alpha-\mathcal{Z}(m_I)}.
\end{align*}
\end{proposition}
\begin{proof}
From the definition of the Fourier transform and  the orthogonality of the additive characters (Proposition \ref{Orthogonality}),  it follows that
\begin{align}
\widehat{S_{k}^t} (m) &= q^{-d}\sum_{x\in\Fq{d}} \chi(-m\cdot x)S_{k}^t(x) \nonumber\\
&= q^{-d}\sum_{x\in\Fq{d}} \chi(-m\cdot x)\left(q^{-1}\sum_{s\in\Fq{}}\chi(s(\knorm{k}{x}-t))\right)\label{smFT}. 
\end{align}
Now, recall from \eqref{Na} that $\mathbb F_q^d= \bigsqcup\limits_{\alpha\in [d]\cup \{0\}} N_\alpha$. By the definition of the $k$-norm, \eqref{smFT} can be rewritten:
\begin{align}
\widehat{S_{k}^t} (m)& = q^{-d}\sum_{\alpha=0}^{k-1} \sum_{x\in N_\alpha} \chi(-m\cdot x)\left(q^{-1}\sum_{s\in\Fq{}}\chi(s(\knorm{}{x}-t))\right)  \label{spm1} \\
&+ q^{-d}\sum_{\alpha=k}^{d} \sum_{x\in N_\alpha} \chi(-m\cdot x)\left(q^{-1}\sum_{s\in\Fq{}}\chi(-st)\right).\label{spm2}
\end{align}
   Since $t\ne 0,$  \eqref{spm2} vanishes by the orthogonality of $\chi.$ Hence, by rewriting the \eqref{spm1}, we obtain
\begin{equation}\label{sm3}
\begin{aligned}
&\widehat{S_{k}^t} (m)=q^{-d-1}\sum_{s\in \Fq{}}\chi(-st)\sum_{\alpha=0}^{k-1} \sum_{x\in N_\alpha}\chi(s\norm{x}-m\cdot x).
\end{aligned}
\end{equation}
Using Lemma \ref{Lem3.1} to calculate the summation over $x\in N_\alpha$,  we obtain 
\begin{align*}
\widehat{S_{k}^t} (m) &=q^{-d-1}\sum_{\alpha = 0}^{k-1}\sum_{s\in\Fq{*}}\chi(-st)\sum_{\substack{{I \subset [d]}\\{:|I| = d-\alpha}}} \prod_{i \in I}\bigg(\eta(s)\mathcal{G}_1\chi\left(-\frac{{m_i}^2}{4s}\right)-1\bigg)\\
&+q^{-d-1}\sum_{\alpha = 0}^{k-1}\sum_{\substack{{I \subset [d]}\\{:|I| = d-\alpha}}}(q-1)^{\mathcal{Z}(m_{I})}(-1)^{d-\alpha-\mathcal{Z}(m_I)}.
\end{align*}
By the definition of $A(m,t)$ and $B(m),$ the proof is complete. 
\end{proof}

\section{Proof of main theorem}\label{Ch4}   
In this section, we provide the proof of our main result (Theorem \ref{first}) that can be restated as follows.
\begin{theorem} \label{first*}
If $E \subset \mathbb{F}_q^d$ with $d \ge 2$, and 
$
|E| \ge C  q^{\max\left\{ \frac{d+1}{2},\ d - k \right\}},
$ then $D_k(E) = \mathbb{F}_q$.
\end{theorem}
\begin{proof}
Let \( d_E : E \times E \to \mathbb{F}_q \) denote the $k$-distance on $E$ defined by $
d_E(x,y) = \|x - y\|_{k}$ for $x,y \in E$.

Note that \( d_E \) is surjective if and only if for each \( t \in \mathbb{F}_q \), the set \( d_E^{-1}(t) \) is nonempty, or equivalently, the collection \( \{d_E^{-1}(t)\}_{t \in \mathbb{F}_q} \) forms a partition of \( E \times E \).

For $E \subset \Fq{d}$ and $t \in \Fq{}$, let $\nu_{E}(t)$ be the cardinality of the set $d_E^{-1}(t)$, i.e., $\nu_E(t)=\left|d_E^{-1} (t)\right|$. Then, to prove Theorem \ref{first*}, it suffices to show that $\nu_E(t) > 0$ for all  $t \in \Fq{}$. It is obvious that $\nu_E(0) > 0$ since $E$ is nonempty. Now it remains to  show  $\nu_E(t) > 0$
 for all nonzero $t \in \Fq{*}$.
By definition, $\nu_E(t)$ can be written
\begin{equation}\label{Nu}
\nu_{E}(t) = \sum_{\substack{{x,y\in E}\\{:\knorm{k}{x-y}=t}}}1 = \sum_{x,y\in E}S_{k}^t(x-y).
\end{equation}
Then we use  the Fourier inversion formula and the definition of the Fourier transform to obtain
\begin{align}\label{Nu1}
    \nu_E(t) & = \sum_{x,y \in \Fq{d}}E(x)E(y)\bigg{(}\sum_{m \in \Fq{d}}\chi(m\cdot(x-y))\widehat{S_{k}^t}(m)\bigg{)} \notag\\
    & = \sum_{m \in \Fq{d}} \widehat{S_{k}^t}(m) \sum_{x,y \in \Fq{d}}\overline{\chi(-m\cdot x)E(x)}\chi(-m\cdot y)E(y)\\
    & = q^{2d}\sum_{m\in\Fq{d}}\widehat{S_{k}^t}(m)|\widehat{E}(m)|^2.\notag
\end{align}
Invoking Proposition \ref{FSF},  we see that
 \begin{align}\label{AB}
\nu_{E}(t) &=q^{d-1}\sum_{m\in \Fq{d}}\pnorm{}{\widehat{E}(m)}^2\bigg(A(m,t)+B(m)
\bigg).
\end{align}
%where $A(m, t)$ and $B(m)$ were defined in Proposition \ref{FSF}.

Now we  require the following lemma, a key ingredient in deriving our main result,   contains  technically challenging terms that do not appear in the original MS problem.

\begin{lemma}\label{main} Let $A(m,t)$, and $B(m)$ be defined as in Proposition \ref{FSF}.  Then for every $E \subset \Fq{d}$, the following inequalities hold.
\begin{enumerate}
[label=\textup{(\roman*)}, ref=\textup{(\roman*)}]
\item \label{ABC1} $\left|\displaystyle{\sum_{m\in \Fq{d}}\pnorm{}{\widehat{E}(m)}^2}A(m,t)\right| \lesssim q^{-\frac{d-1}{2}}|E|,$
\item \label{ABC2} $\displaystyle{\sum_{m\in \Fq{d}}\pnorm{}{
\widehat{E}(m)}^2}B(m) \gtrsim q^{-d}|E|^2 - q^{-k}|E|$.
\end{enumerate}
\end{lemma}

With Lemma \ref{main} in hand, let us complete the proof of Theorem \ref{first*}. A proof of Lemma \ref{main} will be given in the next section.
The combination of \eqref{AB} and Lemma \ref{main} yields 
\begin{equation*}
	\nu_E(t) \gtrsim q^{-1}|E|^2 -q^{d-k-1}|E|-q^{\frac{d-1}{2}}|E|.
\end{equation*}
Observe that if \( |E| \gtrsim q^{d-k} \), then the term \( q^{-1}|E|^2 \) dominates \( q^{d-k-1}|E| \); similarly, if \( |E| \gtrsim q^{(d+1)/2} \), then \( q^{-1}|E|^2 \) dominates \( q^{(d-1)/2}|E| \). Thus, we obtain the desired conclusion that 
if $|E| \ge C  q^{\max\left\{ \frac{d+1}{2},\ d - k \right\}},$ then  $\nu_E(t) >0.$
This completes the proof.
\end{proof}

\section{Proof of Lemma \ref{main} }\label{Ch5}

For the first part \ref{ABC1} of Lemma  \ref{main}, we write
\begin{align*}
A(m,t)=\sum_{\alpha = 0}^{k-1}\sum_{\substack{{I \subset [d]}\\{:|I| = d-\alpha}}}\sum_{s \neq 0}\chi(-st)&\prod_{i \in I}\bigg(\eta(s)\mathcal{G}_1\chi\bigg{(}-\frac{{m}_i^2}{4s}\bigg{)}-1\bigg)\\
=\sum_{\alpha = 0}^{k-1}\sum_{\substack{{I \subset [d]}\\{:|I| = d-\alpha}}}\sum_{s\neq 0} \chi(-st)&\bigg{(} \sum_{\beta=0}^{d-\alpha}\sum_{\substack{{J \subset I}\\{:|J| = d-\alpha-\beta}}} \eta^{d-\alpha-\beta}(s)\mathcal{G}_1^{d-\alpha-\beta}\chi\bigg{(}-\frac{\norm{m_J}}{4s}\bigg{)}(-1)^{\beta}\bigg{)},\label{Amt:second}
\end{align*}
where  we express the product over $i \in I$ in terms of the summation over $\beta.$ It follows by the triangle inequality that 
\begin{equation*}
\pnorm{}{A(m,t)} \leq \sum_{\alpha = 0}^{k-1}\sum_{\substack{{I \subset [d]}\\{:|I| = d-\alpha}}} \sum_{\beta=0}^{d-\alpha}\sum_{\substack{{J \subset I}\\{:|J| = d-\alpha-\beta}}}\pnorm{}{\mathcal{G}_1}^{d-\alpha-\beta}\left|\sum_{s\neq 0}  \eta^{d-\alpha-\beta}(s)\chi\bigg{(}-st-\frac{\norm{m_J}}{4s}\bigg{)}\right|.
\end{equation*}
Recall from \eqref{GSQETA} that \( |\mathcal{G}_1| = \sqrt{q} \), and note that the sum over \( s \in \mathbb{F}_q^* \) is a generalized Kloosterman sum, which is dominated by \( 2\sqrt{q} \) as shown in \eqref{KSTM}. Hence,  for every $m\in \mathbb F_q^d$ and $t\in \mathbb F_q^*,$  we have
$$ |A(m, t)| \leq  \sum_{\alpha = 0}^{k-1}\sum_{\substack{{I \subset [d]}\\{:|I| = d-\alpha}}} \sum_{\beta=0}^{d-\alpha}\sum_{\substack{{J \subset I}\\{:|J| = d-\alpha-\beta}}} 2 q^{\frac{d-\alpha-\beta+1}{2}} \lesssim q^{\frac{d+1}{2}},$$
where the last inequality follows from the simple observation that  $q^{\frac{d+1}{2}}$ dominates  $q^{\frac{d-\alpha-\beta+1}{2}}$ for all $\alpha, \beta\ge 0.$  Using this upper bound of $|A(m, t)|$ and the Plancherel theorem,  the part \ref{ABC1} of Lemma \ref{main} is proven as follows:
$$\left|\displaystyle{\sum_{m\in \Fq{d}}\pnorm{}{\widehat{E}(m)}^2}A(m,t)\right| \leq
\displaystyle{\sum_{m\in \Fq{d}}\pnorm{}{\widehat{E}(m)}^2}|A(m,t)| \lesssim  q^{\frac{d+1}{2}} \sum_{m\in \mathbb F_q^d} |\widehat{E}(m)|^2  = q^{-\frac{d-1}{2}} |E|. $$

Now, we prove the part \ref{ABC2} of Lemma \ref{main}. Recall that 

\begin{equation*}
\sum_{m\in \Fq{d}}\pnorm{}{\widehat{E}(m)}^2B(m) = \sum_{m\in \Fq{d}}\pnorm{}{\widehat{E}(m)}^2\sum_{\alpha = 0}^{k-1}\sum_{\substack{{I \subset [d]}\\{:|I| = d-\alpha}}}(q-1)^{\mathcal{Z}(m_{I})}(-1)^{d-\alpha-\mathcal{Z}(m_I)}. 
\end{equation*}

We observe that the summation over $\alpha$ can be rewritten:
\begin{align*}
    &
    \sum_{m \in \mathbb{F}_q^d} |\widehat{E}(m)|^2 \sum_{\alpha = 0}^{k-1} \sum_{\substack{I \subset [d] \\ :|I| = d - \alpha}} (q - 1)^{\mathcal{Z}(m_I)} (-1)^{d - \alpha - \mathcal{Z}(m_I)} \\
    &
    =
    \sum_{m \in \mathbb{F}_q^d} |\widehat{E}(m)|^2 \left(\sum_{\alpha = 0}^{d} - \sum_{\alpha = k}^{d}\right) \sum_{\substack{I \subset [d] \\ :|I| = d - \alpha}} (q - 1)^{\mathcal{Z}(m_I)} (-1)^{d - \alpha - \mathcal{Z}(m_I)} \\
    &=: B_{main}+B_{auxiliary}. 
\end{align*} 
This decomposition allows us to analyze the summation over $\alpha$ in terms of the full summation over 
\( 0 \leq \alpha \leq d \) and the summation over  \( k \leq \alpha \leq d \). It is obvious that
\begin{equation}\label{mineq}
        \sum_{m\in \Fq{d}}\pnorm{}{\widehat{E}(m)}^2B(m) \geq B_{main}-|B_{auxiliary}|.
\end{equation}
We can simply bound the auxiliary term as follows:
\begin{equation}\label{aux}
    |B_{auxiliary}| \lesssim 
    q^{-d}|E|q^{d-k} = q^{-k}|E|.    
\end{equation}

Next, let us estimate a lower bound of the term $B_{main}.$  For each $I\subset [d]$ and $m\in \mathbb F_q^d$, let $\beta:=\mathcal{Z}(m_I)$. Then  $0\leq \beta \leq |I|$. Using a new variable $\beta$
 we can write

\begin{align*}
B_{\text{main}} 
&=
\sum_{m \in \mathbb{F}_q^d} |\widehat{E}(m)|^2 \sum_{\alpha=0}^{d} \sum_{\substack{I \subset [d] \\ :|I| = d-\alpha}} (q-1)^{\mathcal{Z}(m_I)} (-1)^{d-\alpha-\mathcal{Z}(m_I)}\\
&= 
\sum_{\alpha=0}^{d} \sum_{\beta=0}^{d-\alpha} \sum_{\substack{I \subset [d] \\ :|I|=d-\alpha}} \sum_{\substack{m \in \mathbb{F}_q^d \\ :\mathcal{Z}(m_I)=\beta}} |\widehat{E}(m)|^2 (q-1)^\beta (-1)^{d-\alpha-\beta}.
\end{align*}
By the definition of the indicator function $\mathds{1}_{\{0,1, \cdots, d-\alpha\}} (\beta)$,  the above term $B_{\text{main}}$ can be rewritten as follows:
$$B_{\text{main}}=  \sum_{\beta=0}^{d} \sum_{\alpha=0}^{d}\sum_{\substack{I \subset [d] \\ :|I|=d-\alpha}} \sum_{\substack{m \in \mathbb{F}_q^d \\ :\mathcal{Z}(m_I)=\beta}}|\widehat{E}(m)|^2 (q-1)^\beta (-1)^{d-\alpha-\beta} \mathds{1}_{\{0,1, \cdots, d-\alpha\}} (\beta).$$
%where we interchange the order of summation over the variables \(\alpha\) and \(\beta\).
Fixing the variable \(\beta\), we perform a change of variables from \(\alpha\) to \(r\) using the relation $\alpha = d-\beta-r$. Then we obtain the expression
\begin{align*}    
B_{\text{main}}
&=\sum_{\beta=0}^{d} \sum_{r=0}^{d-\beta} \sum_{\substack{I \subset [d] \\ :|I|=\beta+r}} \sum_{\substack{m \in \mathbb{F}_q^d \\ :\mathcal{Z}(m_I)=\beta}} |\widehat{E}(m)|^2 (q-1)^\beta (-1)^r.
\end{align*}
Note that for  each $\beta$ and $I\subset [d],$ the set \ensuremath{\{m \in \mathbb F_q^d ~|~ \mathcal{Z}(m_I)=\beta\} } is a union of subsets  $\{m \in \mathbb F_q^d ~|~ \mathcal{Z}(m_I)=\beta, \mathcal{Z}(m)=w \}$ for $0\leq w \leq d$, and that \( r:=|I|-\beta \leq d - w \).
Thus, using a variable $w$,
$B_{\text{main}}$ can be written
\[
B_{\text{main}}
=
\sum_{w = 0}^d \sum_{\beta = 0}^{w} \sum_{r = 0}^{d - w} 
\sum_{\substack{I \subset [d] \\ :|I| = \beta + r }} 
\sum_{\substack{m \in \mathbb{F}_q^d  \\ :\mathcal{Z}(m_I) = \beta \\: \mathcal{Z}(m) = w}} 
|\widehat{E}(m)|^2 (q - 1)^{\beta} (-1)^r.
\]
By distinguishing between the cases \(\mathcal{Z}(m) = d\) (that is, \(m = (0, \dots, 0)\)) and \(\mathcal{Z}(m) < d\), we have
\begin{align*}
    B_{main} 
    &=
    \sum_{w = 0}^{d-1} \sum_{\beta = 0}^{w} \sum_{r = 0}^{d - w} \sum_{\substack{I \subset [d] \\ :|I| = \beta + r }} \sum_{\substack{{m \in \mathbb{F}_q^d}  \\  { :\mathcal{Z}(m_I) = \beta} \\ {:\mathcal{Z}(m) = w}}} |\widehat{E}(m)|^2 (q - 1)^{\beta} (-1)^r\\
    &+
    \sum_{\beta = 0}^d\sum_{\substack{I \subset [d] \\ :|I| = \beta}} \sum_{\substack{{ m = (0,\cdots,0)}  \\  { :\mathcal{Z}(m_I) = \beta}}} |\widehat{E}(m)|^2 (q - 1)^{\beta}.
\end{align*}
Now we divide the summation over $\beta$ in the first series of summations into two cases: $\beta < w$ and $w=\beta$.
Then we obtain
\begin{align*}
    B_{main}
    &=
    \sum_{w = 0}^{d-1} \sum_{\beta = 0}^{w-1} \sum_{r = 0}^{d - w} \sum_{\substack{I \subset [d] \\ :|I| = \beta + r }} \sum_{\substack{{m \in \mathbb{F}_q^d}  \\  { :\mathcal{Z}(m_I) = \beta} \\ {:\mathcal{Z}(m) = w}}} |\widehat{E}(m)|^2 (q - 1)^{\beta} (-1)^r \\
    &+
    \sum_{w = 0}^{d-1} \sum_{r = 0}^{d - w} \sum_{\substack{I \subset [d] \\ :|I| = w + r }} \sum_{\substack{{m \in \mathbb{F}_q^d}  \\  { :\mathcal{Z}(m_I) = w} \\ {:\mathcal{Z}(m) = w}}} |\widehat{E}(m)|^2 (q - 1)^{w} (-1)^r\\
    &+
    \sum_{\beta = 0}^d\sum_{\substack{I \subset [d] \\ :|I| = \beta}} \sum_{\substack{{ m = (0,\cdots,0)}  \\  { :\mathcal{Z}(m_I) = \beta}}} |\widehat{E}(m)|^2 (q - 1)^{\beta}\\
    &=: B_{m1} + B_{m2}+B_{m3}.
\end{align*}
Hence, a lower bound of $B_{\text{main}}$ can be expressed as follows:
\begin{equation}\label{mineq2}
    B_{\text{main}} \geq -|B_{m1}| + B_{m2} +B_{m3}.
\end{equation}
Notice that $|N_{d-\alpha}|=(q-1)^\alpha\sim q^\alpha$  and 
 by \eqref{PL2} that $\sum\limits_{m \in N_{d-\alpha}} |\widehat{E}(m)|^2 \lesssim q^{-2d+\alpha}|E|^2$. It follows that
\begin{align}
|B_{m1}| &= 
\sum_{w = 0}^{d-1} \sum_{\beta = 0}^{w-1} \sum_{r = 0}^{d - w} \sum_{\substack{I \subset [d] \\ :|I| = \beta + r }} \sum_{\substack{{m \in \mathbb{F}_q^d}  \\  { :\mathcal{Z}(m_I) = \beta}\\{:\mathcal{Z}(m) = w}}} |\widehat{E}(m)|^2 (q - 1)^{\beta}
\nonumber
\\
&\label{err}\lesssim  
q^{-2d+(d-w)}|E|^2q^{w-1} =q^{-d-1}|E|^2.
\end{align}
Recall that
\begin{equation*}
    B_{m2} = 
    \sum_{w = 0}^{d-1} \sum_{r = 0}^{d - w} \sum_{\substack{I \subset [d] \\ :|I| = w + r }} \sum_{\substack{{m \in \mathbb{F}_q^d}  \\  { :\mathcal{Z}(m_I) = w} \\ {:\mathcal{Z}(m) = w}}} |\widehat{E}(m)|^2 (q - 1)^{w} (-1)^r.
\end{equation*}
 Now we show that $B_{m2} = 0$. The main idea for this is that for each $1 \leq w \leq d-1$, $J \subset [d]$, $m \in \mathcal{F}_{[d]\setminus J}$, $B_{m2} $ will be manipulated to involve the summations
 \begin{equation}\label{Zerosum}
 \sum_{r = 0}^{d-w}\sum_{\substack{{T \subset [d]\setminus J}\\{:|T| = r}}}(-1)^{r},
 \end{equation}
 which  will be turned into 
 \[
 \sum_{r = 0}^{d-w}(-1)^{r}(1)^{d-w-r}\binom{d-w}{r}=0.
 \]
To be precise, first changing the order of summations, we write
\begin{equation*}
   B_{m2}=
    \sum_{w = 0}^{d-1} \sum_{\substack{m \in \mathbb{F}_q^d\\ {:\mathcal{Z}(m) = w}}}\sum_{r = 0}^{d-w}\sum_{\substack{I \subset [d]\\{:|I| = w + r} \\ {:\mathcal{Z}(m_I) = w}}}|\widehat{E}(m)|^2 (q-1)^w(-1)^{r}.
\end{equation*}
Recall from \eqref{Fslice} that 
\begin{equation}\label{FJslice}
\Fq{d} = \bigsqcup_{J \subset [d]} \mathcal{F}_{[d] \setminus J}.
\end{equation}
Thus using a variable $J$,  $B_{m2}$  can be rewritten
\begin{align*}
    B_{m2} &= 
    \sum_{w = 0}^{d-1} \sum_{J \subset [d]}\sum_{\substack{{m \in \mathbb{F}_q^d} \\ {:\mathcal{Z}(m) = w} \\ {:m \in \mathcal{F}_{[d]\setminus J}}}}\sum_{r = 0}^{d-w}\sum_{\substack{I \subset [d]\\{:|I| = w + r} \\ {:\mathcal{Z}(m_I) = w}}}|\widehat{E}(m)|^2 (q-1)^w(-1)^{r}.
\end{align*}
For a fixed $m \in \mathcal{F}_{[d]\setminus J}$, $\mathcal{Z}(m) = w$ if and only if $|J| = w$. Consequently, we obtain
\begin{align}\label{IT}
    B_{m2} &=
    \sum_{w = 0}^{d-1} \sum_{\substack{{J \subset [d]}\\{:|J| = w}}}\sum_{\substack{{m \in \mathbb{F}_q^d} \\ {:m \in \mathcal{F}_{[d]\setminus J}}}}\sum_{r = 0}^{d-w}\sum_{\substack{I \subset [d]\\{:|I| = w + r} \\ {:\mathcal{Z}(m_I) = w}}}|\widehat{E}(m)|^2 (q-1)^w(-1)^{r}.
\end{align}
For each $J\subset [d]$, choose $m\in \mathcal{F}_{[d]\setminus J}$. Then there is an $I$ with
$|I|\geq |J|$ such that  $  {\mathcal{Z}(m_I) = |J|}$, and for any such $I$ (independent of the choice of $m\in \mathcal{F}_{[d]\setminus J}$), we have $J\subset I.$  Let $T:=I\setminus J.$ Changing the fifth summation over $I$ in \eqref{IT} into the summation over $T$, we obtain
\begin{equation}\label{Tbinom}
\begin{aligned}
    B_{m2} &= 
    \sum_{w = 0}^{d-1} \sum_{\substack{{J \subset [d]}\\{:|J| = w}}}\sum_{\substack{{m \in \mathbb{F}_q^d} \\ {:m \in \mathcal{F}_{[d]\setminus J}}}}|\widehat{E}(m)|^2 (q-1)^w\sum_{r = 0}^{d-w}\sum_{\substack{{T \subset [d]\setminus J}\\{:|T| = r}}}(-1)^{r}.
\end{aligned}
\end{equation}
The last double summation  of \eqref{Tbinom} can be turned into 
\begin{equation*}
\begin{aligned}
    \sum_{r = 0}^{d-w}\sum_{\substack{{T \subset [d]\setminus J}\\{:|T| = r}}}(-1)^{r}=\sum_{r = 0}^{d-w}(-1)^{r}\left( \sum_{\substack{{T \subset [d]\setminus J}\\{:|T| = r}}}1 \right)=\sum_{r = 0}^{d-w}(-1)^{r}(1)^{d-w-r}\binom{d-w}{r}.
\end{aligned}
\end{equation*}
By the binomial theorem, we have
\begin{equation*}
    \sum_{r = 0}^{d-w}(-1)^{r}(1)^{d-w-r}\binom{d-w}{r}= (1+(-1))^{d-w} =0.
\end{equation*}
Therefore, we obtain the desired result:
\begin{equation}\label{can}
B_{m_2} = 0. 
\end{equation}
Finally, we estimate the term $B_{m3}$.  Recall that
\begin{equation*}
    B_{m3} = \sum_{\beta = 0}^d\sum_{\substack{I \subset [d] \\ :|I| = \beta}} \sum_{\substack{{ m = (0,\cdots,0)}  \\  { :\mathcal{Z}(m_I) = \beta}}} |\widehat{E}(m)|^2 (q - 1)^{\beta}.
\end{equation*}
Since $|\widehat{E}(m)|^2 = q^{-2d} |E|^2$ for \( m = (0, \ldots, 0) \), we get
\begin{equation}\label{emp}
    B_{m3} = 
    \sum_{\beta = 0}^d\sum_{\substack{I \subset [d] \\ :|I| = \beta}} q^{-2d}|E|^2(q - 1)^{\beta} \sim q^{-d}|E|^2.
\end{equation}
Combining the estimates in \eqref{mineq}, \eqref{aux}, \eqref{mineq2}, \eqref{err}, \eqref{can}, and \eqref{emp}, we conclude that

\begin{equation*}
    \sum_{m \in \mathbb{F}_q^d} |\widehat{E}(m)|^2 B(m) \gtrsim q^{-d}|E|^2 - q^{-d-1}|E|^2 - q^{-k}|E| \gtrsim q^{-d}|E|^2 - q^{-k}|E|,
\end{equation*}
which proves the second part \ref{ABC2} of Lemma \ref{main}.

\end{document}